\documentclass{article}
\usepackage{gastex}
\usepackage{tikz}
\usepackage{hyperref}
\usepackage{enumitem}
\usepackage{amsmath,amscd}
\usepackage{amssymb}
\usepackage{upref}
\usepackage{multirow}
\usepackage{multicol}
\usepackage{url}
\usepackage[all]{xy}
\usepackage{comment}

\usepackage{color}
\usepackage{lineno}
\usepackage{makeidx}
\makeindex
\usepackage{theorem}
\newtheorem{theorem}{Theorem}
\newtheorem{lemma}[theorem]{Lemma}

\newtheorem{corollary}[theorem]{Corollary}
{\theorembodyfont{\rmfamily}%
\newtheorem{example}[theorem]{Example}
\newtheorem{remark}[theorem]{Remark} }
\newenvironment{proof}{\noindent\textit{Proof.}}
{\QED\vskip\theorempostskipamount} 
\newenvironment{proofof}[1]{\noindent\textit{Proof
\protect{#1}.}}
{\QED\vskip\theorempostskipamount}
\def\petitcarre{\vrule height4pt width 4pt depth0pt}
\def\QED{\relax\ifmmode\eqno{\hbox{\petitcarre}}\else{%
\unskip\nobreak\hfil\penalty50\hskip2em\hbox{}\nobreak\hfil
\petitcarre
\parfillskip=0pt \finalhyphendemerits=0\par\smallskip}
\fi}

\newcommand\A{\mathcal{A}}

\def\G{\mathcal{G}}

\newcommand\cL{\mathcal{L}}

\newcommand{\Z}{\mathbb{Z}}
\newcommand{\R}{\mathbb{R}}
\def\un(#1){\underline{#1}\,}

\usetikzlibrary{decorations.pathmorphing}

\DeclareMathOperator{\rank}{rank}

\def\Im{\text{\upshape{Im}}}

\definecolor{ivoire}{rgb}{0.99,0.99,0.8}

\DeclareMathOperator{\Sup}{Sup}

\definecolor{light-gray}{gray}{0.7}

\usepackage{calc}
\newcounter{hours}\newcounter{minutes}

\numberwithin{theorem}{section}
\numberwithin{equation}{section}
\numberwithin{figure}{section}
\numberwithin{table}{section}

\definecolor{lime}{HTML}{A6CE39}
\DeclareRobustCommand{\orcidicon}{%
	\begin{tikzpicture}
	\draw[lime, fill=lime] (0,0)
	circle [radius=0.16]
	node[white] {{\fontfamily{qag}\selectfont \tiny ID}};
	\draw[white, fill=white] (-0.0625,0.095)
	circle [radius=0.007];
	\end{tikzpicture}
	\hspace{-2mm}
}
\foreach \x in {A, ..., Z}{%
\expandafter\xdef\csname orcid\x\endcsname{\noexpand%
\href{https://orcid.org/\csname orcidauthor\x\endcsname}{\noexpand\orcidicon}}
}

\title{Sofic measures}

\author{Marie-Pierre B\'eal\orcidA{}, Eugene Gilburg,
   Vincent Jug\'e, \\ Jean Mairesse,
and Dominique Perrin}
\begin{document}

\maketitle

\begin{abstract}
Sofic measures, also known as hidden Markov measures, have been extensively studied.
In this paper, we survey some equivalent definitions of this notion and improve a bound
for deciding whether a sofic measure is~$k$-step Markov. We prove that if an invariant sofic measure with a linear representation of dimension~$n$ is a~$k$-step Markov chain, then~$k$ can be chosen at most equal to~$\binom{n}{2}$ and that this bound is tight for sofic measures over an alphabet of size $3$.
\end{abstract}

\section{Introduction}
Sofic measures are also known as hidden Markov measures, and have been widely studied.
They are of frequent use in applications.
They were studied in~\cite{HanselPerrin1989} under the name of rational measures, and defined as a particular class of rational series.
Several equivalent definitions are given in the survey of Boyle and Petersen~\cite{BoylePetersen2011}, where the term of sofic measure is introduced.
A result due to Heller~\cite{Heller1965} characterizes the sofic measures that are~$k$-step Markov.
A bound established in~\cite{BoylePetersen2011} shows that the problem of whether a sofic measure is~$k$-step Markov is decidable.
In this paper, we improve significantly this bound (Theorem~\ref{theoremNew2}).

We prove the $\binom{n}{2}$ bound using exterior algebra. We prove that the bound is tight for sofic measures over an alphabet of size at least $3$.
Another approach gives a proof reducing the bound of \cite{BoylePetersen2011} to $2^{n^2}$
using a result of extremal set theory~\cite{HegedusFrankl2024}, itself a variation of a result of Lov\'asz~\cite{Lovasz1977}. 

The paper is organized as follows.
A preliminary section contains definitions from symbolic dynamics and probability theory.
Section~\ref{sectionSofic} introduces sofic measures as images of~$1$-step Markov measures by a factor map: this is the traditional definition of a hidden Markov measure.
We prove the equivalence to the definition as a particular class of rational series (Theorem~\ref{theoremHP}), originally due to Furstenberg~\cite{Furstenberg1960}.
In section~\ref{sectionMarkovSofic}, we recall a result due to Holland~\cite{Holland1968} characterizing the sofic measures that are~$k$-step Markov (Theorem~\ref{theoremHolland}).
Section \ref{sectionIdentification} contains our main result (Theorem~\ref{theoremNew2}).
The link with extremal set theory and Bollobas systems
is presented in Section~\ref{sectionBollobas}.
The main results of this paper have been obtained using AI
systems, and we state this explicitly in the last section.
\section{Preliminaries}
Let~$A$ be a finite alphabet.
Let~$A^*$ denote the set of finite words over~$A$.
The empty word is denoted by~$\varepsilon$.

Let~$A^\Z$ denote the set of two-sided infinite sequences on the alphabet~$A$, considered as a compact topological space for the usual topology.
For~$x\in A^\Z$ and~$i\leqslant j$, the factor~$x_i\cdots x_{j}$ is denoted by~$x_{[i,j]}$.
Let~$S$ denote the shift transformation on~$A^\Z$ defined by~$y=S(x)$ if~$y_n=x_{n+1}$ for every~$n\in\Z$.

A \emph{shift space} is a closed set of two-sided
infinite sequences on an alphabet~$A$ that is invariant under the action of~$S$.
The \emph{language} of the shift space~$X$, denoted by~$\cL(X)$, is the set of words appearing as factors of elements of~$X$, that is, the set of words~$x_{[i,j)}$ for which~$x\in X$ and~$i<j$.
For an individual sequence, let~$\cL(x)$ denote the set of factors of~$x$.
Thus,~$\cL(X)=\cup_{x\in X}\cL(x)$.
Let~$\cL_n(X)$ also denote the elements of~$\cL(X)$ of length~$n$.

A shift space is of \emph{finite type} if there is a finite set~$F$ of words such that~$\cL(X)$ is the set of all sequences without any factor in~$F$.

Let~$X$ and~$Y$ be shift spaces on alphabets~$A$ and~$B$, respectively.
Let~$m$ and~$n$ be non-negative integers and let~$f\colon\cL_{m+n+1}\to B$ be a map, called a \emph{block map}.
The \emph{sliding block code} defined by~$f$ is the map~$\phi$ from~$X$ to~$B^\Z$ defined by~$\phi(x)_i=f(x_{i-m}\cdots x_{i+n})$ for all~$i\in\Z$.
We say that~$\phi$ is a~\emph{$k$-block map} if~$k=m+n+1$.

If~$\phi$ maps~$X$ onto~$Y$, it is called a \emph{factor map} from~$X$ to~$Y$.
A one-to-one factor map is called a \emph{conjugacy}.

A \emph{sofic shift space} is the image of a shift space of finite type by a factor map.

Let~$X$ be a shift space.
For~$k\geqslant 1$, let~$f\colon \cL_k(X)\to A_k$ be a bijection from the set~$\cL_k(X)$ onto an alphabet~$A_k$.
The sliding block code~$\gamma_k$ defined by~$y=\gamma_k(x)$ if
\[y_n=f(x_{n-k+1}\cdots x_n)\]
is a conjugacy, called the \emph{$k$\textsuperscript{th} higher block code}, and~$X^{(k)}=\gamma_k(X)$ is called the \emph{$k$-block presentation} of~$X$.

A probability measure on a shift space is a probability measure on the
family of Borel subsets of~$X$ (that is, the~$\sigma$-algebra generated by open sets).
A  probability measure~$\mu$ on a shift space~$X$ is \emph{invariant} if~$\mu(S^{-1}U)=\mu(U)$ for every Borel subset~$U$ of~$X$.

If~$\phi\colon X\to Y$ is a factor map, the measure~$\nu=\phi\mu$, called the \emph{image} of~$\mu$ by~$\phi$, is defined by~$\nu(U)=\mu(\phi^{-1}(U))$ for every Borel subset~$U$ of~$Y$.
If~$\mu$ is invariant, the image of~$\mu$ is invariant.

For~$w\in \cL(X)$, the cylinder defined by a word~$w$ of length~$n$ is the set~$[w]_X=\{x\in X\mid x_{[0,n)}=w\}$.

Let $\mu$ be a probability measure on a shift space $X$.
The \emph{probability distribution} of $\mu$ is the map 
$\pi\colon \cL(X)\to[0,1]$ such that~$\pi(w)=\mu([w]_X)$.
It satisfies
\begin{equation}
\pi(\varepsilon)=1,\label{eqepsilon}
\end{equation}
and, for every~$w\in\cL(X)$,
\begin{equation}
\pi(w)=\sum_{a\in A, wa\in \cL(X)}\pi(wa).\label{eqpi}
\end{equation}
If~$\mu$ is invariant, it also satisfies
\begin{equation}
\pi(w)=\sum_{a\in A, aw\in\cL(X)}\pi(aw).\label{eqpiR}
\end{equation}
Conversely, for every~$\pi$ satisfying \eqref{eqepsilon},
\eqref{eqpi}
and \eqref{eqpiR} for every~$w\in \cL(X)$, there is a unique invariant probability measure~$\mu$ on~$X$ such $\pi$ is the probability distribution of $\mu$.
We denote $\hat{\mu}=\pi$.

As a particular case, a \emph{Bernoulli measure} on~$A^\Z$ is a probability measure~$\mu$ such that~$\hat{\mu}$ is a monoid morphism from~$A^*$ into~$[0,1]$.

If~$\mu$ is a probability mesure on a shift space~$X$, its \emph{support}~$\Sup(\mu)$ is the set of sequences~$x$ such that~$\hat{\mu}(w)>0$ for all~$w\in\cL(x)$.
It is a shift space and~$\mu$ defines a probability measure on~$\Sup(\mu)$.

Given a finite set~$A$, an irreducible~$A\times A$-stochastic matrix~$M$ and a stochastic row vector~$v\in \R^A$ such that~$vM=v$, the \emph{$1$-step Markov measure} defined by~$(v,M)$ is the measure~$\mu$ on~$A^\Z$ such that
\begin{displaymath}
\hat{\mu}(a_1a_2\cdots a_n)=v_{a_1}M_{a_1,a_2}\cdots M_{a_{n-1},a_n}.
\end{displaymath}
It is an invariant probability measure on the shift of finite type
\begin{displaymath}
X_M=\{x\in A^\Z\mid M_{x_i,x_{i+1}}>0\text{ for all~$i\in\Z$}\},
\end{displaymath}
and therefore on any shift containing~$X_M$.
When we refer to a~$1$-step Markov measure~$\mu$ on a shift space~$X$, we mean that~$\mu$ is defined by a pair~$(v, M)$ as above and that~$X_M\subseteq X$.

A probability measure~$\mu$ on~$X$ is \emph{$k$-step Markov} if its image under the~$k$\textsuperscript{th} higher block code~$\gamma_k$ is a~$1$-step Markov measure~$\nu$.
We have then
\begin{equation}
\hat{\nu}(\gamma_k(u))=\hat{\mu}(u),\label{eqhatnu}
\end{equation}
for every~$u\in \cL_k(X)$.
\section{Sofic measures}\label{sectionSofic}
A  probability measure~$\nu$ on a shift space~$Y$ is \emph{sofic} if~$\nu=\phi \circ \mu$ for some factor map~$\phi$ from a shift of finite type~$X$, a~$1$-step Markov measure~$\mu$ on~$X$, and a factor map~$\phi\colon X\to Y$.
Since~$Y$ is the image of~$X$ by~$\phi$, it is a sofic shift.
The factor map~$\phi$ can be assumed to be~$1$-block since the image of a~$1$-step Markov
measure on~$X$ in the~$k$-block presentation of~$X$ is also a~$1$-step Markov measure (see the figure below, where~$\phi_k=\phi\circ\gamma_k^{-1}$).
\[
\begin{tikzpicture}
\node(X) at(0,1){$X$};\node(Xk) at(2,1){$X^{(k)}$};\node(Y) at(1,0){$Y$};
\draw[->,left](X)edge node{$\phi$}(Y);\draw[->,above](X)edge node{$\gamma_k$}(Xk);\draw[->,right](Xk)edge node{$\phi_k$}(Y);
\end{tikzpicture}
\]

Since a~$1$-step Markov measure is assumed to be invariant, a sofic measure is invariant.

A sofic measure is also known as a \emph{hidden Markov measure}.
The term \emph{Markov chain} is also of common use instead of Markov measure, just as \emph{hidden Markov chain} is used instead of hidden Markov measure.

Given an integer~$n\geqslant 1$, a row vector~$\lambda\in \R^n$, a column vector~$\gamma\in\R^n$ and a morphism~$\varphi$ is a morphism  from~$A^*$ to the monoid of~$n\times n$ real matrices, the triple~$(\lambda,\varphi,\gamma)$ defines a map~$\pi\colon A^*\to\R$ by
\begin{equation}
\pi(w)=\lambda\varphi(w)\gamma.\label{eqLinRep}
\end{equation}
for~$w\in A^*$. 

The triple~$(\lambda,\varphi,\gamma)$ is said to be a \emph{linear representation} of~$\pi$.
Its \emph{dimension} is the integer $n$.
The representation is non-negative if~$\lambda$, $\gamma$, and the matrices~$\varphi(w)$ are non-negative.
Two triples are equivalent if they define the same map.

The following result is a combination of ideas appearing in~\cite{Furstenberg1960},\cite{Heller1965} and~\cite{HanselPerrin1989} (see~\cite[Theorem 4.20]{BoylePetersen2011}). A linear representation of a probability measure~$\mu$ is a linear representation of the map~$\hat{\mu}$.
\begin{theorem}\label{theoremHP}
Let~$Y$ be a shift space over the alphabet~$A$.
An invariant probability measure on~$Y$ has a non-negative linear representation if and only if it is sofic.
In this case, the linear representation can be chosen such that
\begin{enumerate}[label=\rm(\roman*)]
\item the matrix~$M=\sum_{a\in A}\varphi(a)$ is stochastic;
\item the vector~$\lambda$ is stochastic and~$\gamma$ has all its components equal to~$1$;
\item the vector~$\lambda$ satisfies the relation~$\lambda M=\lambda$.
\end{enumerate}

\end{theorem}
\begin{proof}
Let~$\nu=\phi\mu$ be a sofic measure on~$Y$, with~$\phi\colon X\to Y$ a factor map and~$\mu$ a 1-step Markov on the shift of finite type~$X$. Let~$(v,M)$ be the pair of a vector
and a~$B\times B$-stochastic matrix defining the Markov measure~$\mu$ on~$X=X_M$.
We may assume that the factor map~$\phi$ is~$1$-block, and therefore defined by a block map~$f\colon B\to A$.
Consider the graph on~$B$ with edges~$(p,a,q)$ when~$f(q)=a$.
The measure~$\nu$ is then defined by
\begin{displaymath}
\hat{\nu}(w)=\sum_{\text{$c$ a path labelled by~$w$}}\hat{\mu}(c).
\end{displaymath}
We obtain a linear representation of~$\nu$ by setting~$\lambda=v$, taking~$\gamma$ as the column vector of all ones, and defining~$\varphi(w)$ as the~$B\times B$-matrix with coefficients
\begin{displaymath}
\varphi_{p,q}(a)=\begin{cases}M_{p,q}&\text{ if~$f(q)=a$,}\\
0&\text{otherwise.}\end{cases}
\end{displaymath}
The triple~$(\lambda,\varphi,\gamma)$ satisfies conditions (i), (ii)
and (iii). 

Conversely, assume that~$\mu$ has a non-negative linear representation~$(\lambda,\varphi,\gamma)$.
We first show that we can assume that the triple satisfies conditions (i), (ii), and (iii).

Set~$M=\sum_{a\in A}\varphi(a)$.
We may assume that the representation is trim, that is, that for every index~$q$, there exist integers~$k,\ell\geqslant 0$ such that~$(\lambda M^k)_q>0$ and~$(M^\ell\gamma)_q>0$.

Then we can change~$(\lambda,\varphi,\gamma)$ into~$(\lambda,\varphi,M^k\gamma)$.
Indeed, arguing by induction on~$k$, we have, for every word~$u$,
\begin{displaymath}
\lambda\varphi(u)M^k \gamma
=\sum_{a\in A}\lambda\varphi(u)\varphi(a)M^{k-1}\gamma
=\sum_{a\in A}\hat{\mu}(ua)=\hat{\mu}(u).
\end{displaymath}
Since the representation is trim, the coefficients of~$M^n$ remain bounded:
if~$M^n_{pq}$ is unbounded,~$\lambda M^{k+n+\ell}\gamma$ is unbounded whenever~$\lambda M^k_p>0$ and~$M^\ell_q \mu >0$.
Let~$\gamma'$ be a vector which is the limit of a subsequence of the sequence~$\gamma_k=\frac{1}{k}\sum_{i=1}^kM^k\gamma$.
Then~$M\gamma'=\gamma'$ and~$(\lambda,\varphi,\gamma')$ is a non-negative linear representation of~$\mu$.

We may also assume that~$\gamma'$ is positive;
otherwise, we can simply omit the corresponding indices (if~$\gamma'_q=0$, we have~$(\varphi(w)\gamma')_q=0$ for every word~$w$).
Let~$D$ be the diagonal matrix with coefficients~$D_{i,i}=\gamma'_i$.
The triple~$(\lambda D,\psi,D^{-1}\gamma')$ given by~$\psi(w)=D^{-1}\varphi(w) D$ satisfies conditions (i) and (ii).

Finally, a similar argument shows that we can change~$\lambda$ to a vector~$\lambda'$ such that~$\lambda'M=\lambda'$.

We can now consider the set~$B=\{1,2\ldots,n\}\times A$, where~$n$ is the dimension of the matrices~$\varphi(w)$.
Let~$N$ be the matrix
\begin{displaymath}
N_{p,q}=\begin{cases}\varphi(b)_{i,j}&\text{ if~$p=(i,a)$ and $q=(j,b)$,}\\
0&\text{ otherwise.}
\end{cases}
\end{displaymath}
We have
\begin{displaymath}
\sum_{q\in B}N_{p,q}=\sum_{j=1}^n\sum_{b\in A}\varphi(b)_{i,j}=1,
\end{displaymath}
and thus~$N$ is stochastic.
It is easy to verify that~$\mu$ is the sofic measure defined by the Markov measure corresponding to~$(v, N)$ with~$v$ such that~$\sum_{a\in A}v_{i, a}=\lambda_i$.
\end{proof}
We do not know whether, in Theorem~\ref{theoremHP}, the fact that the measure has a non-negative linear representation is necessary, nor whether there exist invariant probability measures that have a linear representation, but no 
minimal non-negative one.

\begin{example}\label{exampleBoylePetersen1}
Consider the~$1$-step Markov measure defined by the pair 
\begin{displaymath}
v=\begin{bmatrix}1/3&1/3&1/3\end{bmatrix},\quad
M=\begin{bmatrix}0&2/3&1/3\\
2/3& 1/3&0\\
1/3&0&2/3\\
\end{bmatrix}.
\end{displaymath}
The shift of finite type~$X_M$ is represented in Figure~\ref{figureExample} on the left.

Using the~$1$-block map~$f(1)=a$ and~$f(2)=f(3)=b$, we obtain an invariant sofic measure defined by the diagram of Figure~\ref{figureExample} on the right.

\begin{figure}[hbt]
\centering
\tikzset{node/.style={circle,draw,minimum size=0.4cm,inner sep=0.4pt}}
\tikzset{box/.style={draw,minimum size=0.4cm,inner sep=0pt}}
\tikzstyle{loop right}=[in=-40,out=40,loop]
\begin{tikzpicture}
\node[node](1)at(-5,0){$1$};\node[node](2)at(-2.5,1.5){$2$};
\node[node](3)at(-2.5,-1.5){$3$};

\draw[->,above,bend left=20](1)edge node{$\frac{2}{3}$}(2);
\draw[->,loop right,right](2)edge node{$\frac{1}{3}$}(2);
\draw[->,above,bend left=20](2)edge node{$\frac{2}{3}$}(1);
\draw[->,above,bend left=20,near end](1)edge node{$\frac{1}{3}$}(3);
\draw[->,loop right,right](3)edge node{$\frac{2}{3}$}(3);
\draw[->,above,bend left=20](3)edge node{$\frac{1}{3}$}(1);

\node[node](1)at(0,0){$1$};\node[node](2)at(2.5,1.5){$2$};
\node[node](3)at(2.5,-1.5){$3$};

\draw[->,above,bend left=20](1)edge node[anchor=south,yshift=2pt]{$\frac{2}{3} b$}(2);
\draw[->,loop right,right](2)edge node[anchor=west]{$\frac{1}{3} b$}(2);
\draw[->,above,bend left=20](2)edge node[anchor=west]{$\frac{2}{3} a$}(1);
\draw[->,above,bend left=20,near end](1)edge node[anchor=west]{$\frac{1}{3} b$}(3);
\draw[->,loop right,right](3)edge node[anchor=west]{$\frac{2}{3} b$}(3);
\draw[->,above,bend left=20](3)edge node[anchor=north,yshift=-2pt]{$\frac{1}{3} a$}(1);
\end{tikzpicture}
\caption{A sofic measure.}\label{figureExample}
\end{figure}

The corresponding linear representation is given by
\[\lambda=\begin{bmatrix} 1/3&1/3&1/3\end{bmatrix},\quad\gamma=\begin{bmatrix}1&1&1\end{bmatrix}^t
\]
and
\begin{displaymath}
\varphi(a)=\begin{bmatrix}0&0&0\\2/3&0&0\\1/3&0&0\end{bmatrix},\quad
\varphi(b)=\begin{bmatrix}0&2/3&1/3\\0&1/3&0\\0&0&2/3\end{bmatrix}.
\end{displaymath}
\end{example}

\section{Markov measures vs sofic measures }\label{sectionMarkovSofic}
We now focus on the problem of determining when a sofic probability measure is a~$k$-step Markov measure.

We first need to recall the notion of minimal linear representation.
A linear representation~$(\lambda,\varphi,\gamma)$ of a map $\pi\colon A^* \to \R$ of dimension~$n$ is \emph{reduced} if
\begin{enumerate}[label=(\roman*)]
\item the vectors~$\lambda\varphi(w)$ for~$w\in A^*$ generate~$\R^n$, and
\item the vectors~$\varphi(w)\gamma$ for~$w\in A^*$ generate~$\R^n$.
\end{enumerate}
It is classical (see~\cite{BerstelReutenauer2011} or~\cite{Sakarovitch2009}) that,
up to a linear isomorphism, there is a unique reduced equivalent linear representation.
Moreover, for every linear representation~$(\lambda,\varphi,\gamma)$, there is a reduced linear representation~$(\lambda',\varphi',\gamma')$ such that~$\lambda=\begin{bmatrix}*&\lambda'&0\end{bmatrix}$,~$\gamma=\begin{bmatrix}0&\gamma'&*\end{bmatrix}^t$ and, for every word~$w\in A^*$,
\begin{displaymath}
\varphi(w)=\begin{bmatrix}*&0&0\\ *&\varphi'(w)&0\\ *&*&*\end{bmatrix}.
\end{displaymath}

We give three examples.
In the first one, we again consider the sofic measure of Example~\ref{exampleBoylePetersen1}.

\begin{example}
The sofic measure defined by the representation~$(\lambda,\varphi,\gamma)$
of Example \ref{exampleBoylePetersen1} is reduced.
Indeed, we have
\begin{displaymath}
\lambda=\begin{bmatrix}\frac{1}{3}&\frac{1}{3}&\frac{1}{3}\end{bmatrix},\quad
\lambda\varphi(a)=\begin{bmatrix}\frac{1}{3}&0&0\end{bmatrix},\quad
\lambda\varphi(ab)=\begin{bmatrix}0&\frac{2}{9}&\frac{1}{9}\end{bmatrix},
\end{displaymath}
and thus the vectors~$\lambda\varphi(w)$ generate~$\R^3$.
Symmetrically, we have
\begin{displaymath}
\gamma=\begin{bmatrix}1\\1\\1\end{bmatrix},\quad
\varphi(a)\gamma=\begin{bmatrix}0\\2/3\\1/3\end{bmatrix},\quad
\varphi(b{a})\gamma=\begin{bmatrix}{5/9}\\{2/9}\\{2/9}\end{bmatrix},
\end{displaymath}
and thus the vectors~$\varphi(w)\gamma$ also generate~$\R^3$.
\end{example}

In the second example, we modify the coefficients of the 1-step Markov measure to
obtain a sofic measure with a minimal representation of a smaller dimension.

\begin{example}
Consider now the~$1$-step Markov measure represented in Figure~\ref{figureExample2} on the left.
Using the same~$1$-block factor map as in Example~\ref{exampleBoylePetersen1}, we obtain the sofic measure represented on the right.

\begin{figure}[hbt]
\centering
\tikzset{node/.style={circle,draw,minimum size=0.4cm,inner sep=0.4pt}}
\tikzset{box/.style={draw,minimum size=0.4cm,inner sep=0pt}}
\tikzstyle{loop right}=[in=-40,out=40,loop]
\begin{tikzpicture}[>=stealth]
\node[node](1)at(-5,0){$1$};\node[node](2)at(-2.5,1.5){$2$};
\node[node](3)at(-2.5,-1.5){$3$};

\draw[->,above,bend left=20](1)edge node{$\frac{1}{2}$}(2);
\draw[->,loop right,right](2)edge node{$\frac{1}{2}$}(2);
\draw[->,above,bend left=20](2)edge node{$\frac{1}{2}$}(1);
\draw[->,above,bend left=20,near end](1)edge node{$\frac{1}{2}$}(3);
\draw[->,loop right,right](3)edge node{$\frac{1}{2}$}(3);
\draw[->,above,bend left=20](3)edge node{$\frac{1}{2}$}(1);

\node[node](1)at(0,0){$1$};\node[node](2)at(2.5,1.5){$2$};
\node[node](3)at(2.5,-1.5){$3$};

\draw[->,above,bend left=20](1)edge node{$\frac{1}{2}b$}(2);
\draw[->,loop right,right](2)edge node{$\frac{1}{2}b$}(2);
\draw[->,above,bend left=20](2)edge node{$\frac{1}{2}a$}(1);
\draw[->,above,bend left=20,near end](1)edge node{$\frac{1}{2}b$}(3);
\draw[->,loop right,right](3)edge node{$\frac{1}{2}b$}(3);
\draw[->,above,bend left=20](3)edge node{$\frac{1}{2}a$}(1);
\end{tikzpicture}
\caption{A sofic measure.}\label{figureExample2}
\end{figure}
The corresponding linear representation is~$(\lambda,\varphi,\gamma)$, with
\begin{displaymath}
\varphi(a)=\begin{bmatrix}0&0&0\\1/2&0&0\\1/2&0&0\end{bmatrix},\quad
\varphi(b)=\begin{bmatrix}0&1/2&1/2\\0&1/2&0\\0&0&1/2\end{bmatrix},
\end{displaymath}
$\lambda=\begin{bmatrix}1/3&1/3&1/3\end{bmatrix}$ and~$\gamma=\begin{bmatrix}1&1&1\end{bmatrix}^t$.

We use the algorithm computing the reduced representation, as described in~\cite{BerstelReutenauer2011} or~\cite{Sakarovitch2009}.
The vector space generated by the vectors~$\varphi(w)\gamma$ has dimension two.
Using the basis formed by the columns of the matrix~$L$ below, we obtain the equivalent representation~$(\lambda',\varphi',\gamma')$ with~$\gamma=L\gamma'$,~$\lambda'=\lambda L$ and~$\varphi(w)L=L\varphi'(w)$, given by
\begin{displaymath}
L=\begin{bmatrix}1&0\\0&1\\0&1\end{bmatrix},
\lambda'=\begin{bmatrix}1/3&2/3\end{bmatrix},
\varphi'(a)=\begin{bmatrix}0&0\\1/2&0\end{bmatrix},
\varphi'(b)=\begin{bmatrix}0&1\\0&1/2\end{bmatrix},
\gamma'=\begin{bmatrix}1\\1\end{bmatrix}.
\end{displaymath}
\end{example}

The following example shows that the reduced representation corresponding
to a non-negative linear representation need not be non-negative.

\begin{example}
Let~$\lambda=\frac{1}{\sigma}\begin{bmatrix}1&p&r&q&s(r+q)\end{bmatrix}$,~$\gamma=\begin{bmatrix}1&1&1&1&1\end{bmatrix}^t$, and
\begin{displaymath}
\varphi(a)=
\begin{bmatrix}0&p&0&0&0\\ 0&0&r&0&0\\ r&0&0&0&0\\ 0&0&r&0&0\\r&0&0&0&0
\end{bmatrix},\quad
\varphi(b)=
\begin{bmatrix}0&0&0&q&0\\ s&0&0&0&0\\ 0&0&0&0&s\\ 0&0&0&0&s\\s&0&0&0&0
\end{bmatrix},
\end{displaymath}
with~$p,q,r,s>0$,~$p+q=r+s=1$ and~$\sigma=2+r+s(r+q)$.
An equivalent reduced representation can be obtained as follows.
First, the vectors~$\lambda\varphi(w)$ for~$w\in A^*$ generate the space~$\R^5$.
Next, the vectors~$\varphi(w)\gamma$ generate the orthogonal of the vector~$v=\begin{bmatrix}0&1&1&-1&-1\end{bmatrix}$.
Thus, taking as basis of $\R^5$ the row vectors~$e_1$, $e_2$, $e_3$, $e_4$ and~$e_2+e_3-e_4-e_5$ (where~$e_i$ is the~$i$\textsuperscript{th} canonical basis vector), we obtain
\begin{displaymath}
\varphi(a)\sim\begin{bmatrix}0&p&0&0&0\\0&0&r&0&0\\r&0&0&0&0\\0&0&r&0&0\\
0&0&0&0&0\end{bmatrix},\quad
\varphi(b)\sim\begin{bmatrix}0&0&0&q&0\\s&0&0&0&0\\0&s&s&-s&-s\\0&s&s&-s&-s\\
0&0&0&0&0\end{bmatrix},
\end{displaymath}
where~$\sim$ denotes the conjugacy of matrices.
Thus, since the last row is zero, considering the upper-left corner of dimension~$4$
of these matrices, we obtain a linear representation~$(\lambda',\varphi',\gamma')$ of dimension~$4$ with 
\[\lambda'=\frac{1}{\sigma}\begin{bmatrix}1&p+s(r+q)&r+s(r+q)&q-s(r+q)\end{bmatrix}\]
and~$\gamma'=\begin{bmatrix}1&1&1&1\end{bmatrix}^t$.
This representation is minimal when~$p\neq r$; otherwise, the measure is Bernoulli and the minimal representation
has dimension~$1$.
\end{example}

The following result is from~\cite{Holland1968}.
\begin{theorem}\label{theoremHolland}
An invariant sofic measure is~$k$-step Markov if and only if, in its minimal linear representation,
each matrix~$\varphi(w)$
corresponding to a word~$w$  of length~$k$ is of rank at most 1.
\end{theorem}
We use the following lemma in the proof of Theorem~\ref{theoremHolland}.
\begin{lemma}\label{lemmaJM}
A probability measure~$\mu$ on~$A^\Z$ is~$1$-step Markov if and only if
\begin{displaymath}
\hat{\mu}(abw)=\frac{\hat{\mu}(ab)\hat{\mu}(bw)}{\hat{\mu}(b)}
\end{displaymath}
for every~$a,b\in A$ and~$w\in A^*$.
\end{lemma}
\begin{proof}
If~$\mu$ is the~$1$-step Markov measure defined by the pair~$(v,M)$,
we have for~$a,b\in A$ and~$w=a_1a_2\cdots a_n$,
\begin{eqnarray*}
\hat{\mu}(abw)&=&v_aM_{a,b}M_{b,a_1}M_{a_1,a_2}\cdots M_{a_{n-1},a_n}\\
&=&\frac{(v_aM_{a,b})(v_bM_{b,a_1}\cdots)}{v_b}=\frac{\hat{\mu}(ab)\hat{\mu}(bw)}{\hat{\mu(b)}}.
\end{eqnarray*}
Conversely, set~$v_a=\hat{\mu}(a)$ and 
\begin{displaymath}
M_{a,b}=\frac{\hat{\mu}(ab)}{\hat{\mu}(a)}.
\end{displaymath}
Then
\begin{eqnarray*}
\hat{\mu}(a_1\cdots a_n)
&=&\frac{\hat{\mu}(a_1a_2)\hat{\mu}(a_2\cdots a_n)}{\hat{\mu}(a_2)}
=\hat{\mu}(a_1)\frac{\hat{\mu}(a_1a_2)}{\hat{\mu}(a_1)}
\frac{\hat{\mu}(a_2\cdots a_n)}{\hat{\mu}(a_2)}\\
&=&v_{a_1}M_{a_1,a_2}\frac{\hat{\mu}(a_2\cdots a_n)}{\hat{\mu}(a_2)},
\end{eqnarray*}
proving by induction on~$n$ that~$\hat{\mu}(a_1a_2\cdots a_n)=v_{a_1}M_{a_1,a_2}\cdots M_{a_{n-1},a_n}$.
\end{proof}

\begin{proofof}{of Theorem~\ref{theoremHolland}}
Assume first that~$k=1$.

If~$\mu$ is~$1$-step Markov, let~$(v,M)$ be the pair defining~$\mu$.
Then~$\mu$ has the linear representation~$(\lambda,\varphi,\gamma)$
with~$\lambda=v$,~$\gamma$ with all coefficients~$1$ and
for every~$a,b,c\in A$
\begin{displaymath}
\varphi(a)_{b,c}=\begin{cases}M_{b,c}&\text{ if~$a=c$,}\\
0&\text{otherwise.}
\end{cases}
\end{displaymath}
In this case, all matrices~$\varphi(a)$ have only one non-zero
column and are thus of rank~$1$. This implies that the matrices of the minimal representation have the same property.

Conversely, consider~$a,b\in A$ and~$w\in A^*$. Since~$\varphi(a),\varphi(b)$ have rank~$1$, we can write~$\varphi(a)=c_a\ell_a$,~$\varphi(b)=c_b\ell_b$
where~$c_a,c_b$ are column vectors and~$\ell_a,\ell_b$ are row
vectors. We have, permuting the factors which are scalars,
\begin{eqnarray*}
\hat{\mu}(ab)\hat{\mu}(bw)&=&(\lambda c_a\ell_ac_b\ell_b\gamma)
(\lambda c_b\ell_b \varphi(w)\gamma)\\
&=&\lambda c_a\ell_ac_b(\ell_b\gamma)
(\lambda c_b)(\ell_b \varphi(w)\gamma)\\
&=&\lambda c_a\ell_ac_b(\ell_b \varphi(w)\gamma)
(\lambda c_b)(\ell_b\gamma)\\
&=&\hat{\mu}(abw)\hat{\mu}(b),
\end{eqnarray*}
which shows that~$\mu$ is~$1$-step Markov by Lemma~\ref{lemmaJM}.

In the general case, assume first that~$\mu$ is~$k$-step Markov on~$X$.
Recall that $\gamma_k$ denotes a bijection from~$\cL_k(X)$ onto~$A_k$.
Let~$(v,M)$ be the pair defining a~$1$-step Markov measure on~$X^{(k)}$ such that~$\mu(u)=\hat{\nu}(u)$ for every~$u\in\cL_k(X)$.
Let~$(\lambda,\varphi,\gamma)$ be the linear representation of~$\mu$ defined by~$\lambda=v$,~$\gamma=\begin{bmatrix}1&&\ldots&1\end{bmatrix}^t$ and,
for every~$a\in A$,
\begin{displaymath}
\varphi(a)_{u,v}=\begin{cases}M_{u,v}&\text{ if~$ua=bv$ for some~$b\in A$}\\
0&\text{ otherwise}.\end{cases}
\end{displaymath}
Then, for every~$w\in\cL_k(X)$, the matrix~$\varphi(w)$ has at most one non-zero column, and thus is of rank at most~$1$.

Conversely, let~$(\lambda,\varphi,\gamma)$ be the minimal linear representation of~$\mu$.
Let~$(\lambda,\psi,\gamma)$ be the linear representation defined by~$\psi(\gamma_k(u))=\varphi(u)$ for every~$u\in\cL_k(X)$.
Then all matrices~$\psi(\gamma_k(u))$ have rank at most~$1$, and thus the probability measure~$\nu$ defined by~$\hat{\nu}(\gamma_k(u))=\hat{\mu}(u)$ is~$1$-step Markov, as we have just seen.
This implies that~$\mu$ is~$k$-step Markov.
\end{proofof}

In the next example, the measure is shown to be~$1$-step Markov.
\begin{example}
Let~$\varphi$ be the linear representation given by
\[
\varphi(a)=\begin{bmatrix}2/3&0&0&0\\0&0&0&0\\2/3&0&0&0\\0&0&0&0\end{bmatrix}
,\quad
\varphi(b)=\begin{bmatrix}0&1/3&0&0\\0&0&0&0\\0&0&0&0\\0&1/3&0&0\end{bmatrix},\]
\[
\varphi(c)=\begin{bmatrix}0&0&0&0\\0&0&1/3&0\\0&0&0&0\\0&0&1/3&0\end{bmatrix},\quad
\varphi(d)=\begin{bmatrix}0&0&0&0\\0&0&0&2/3\\0&0&0&1/3\\0&0&0&1/3\end{bmatrix},
\]
$\lambda=\begin{bmatrix}1/3&5/24&1/6&7/24\end{bmatrix}$ and~$\gamma=\begin{bmatrix}1&1&1&1\end{bmatrix}^t$.
The vectors~$\varphi(w)\gamma$ generate the orthogonal of~$\begin{bmatrix}1&1&-1&-1\end{bmatrix}$.
Thus, taking as basis of $\R^4$ the row vectors~$e_1$, $e_2$, $e_3$ and~$e_1+e_2-e_3-e_4$, we obtain the reduced representation
\[
\varphi'(a)=\begin{bmatrix}2/3&0&0\\0&0&0\\2/3&0&0\end{bmatrix},\quad
\varphi'(b)=\begin{bmatrix}0&1/3&0\\0&0&0\\0&0&0\end{bmatrix},\]
\[
\varphi'(c)=\begin{bmatrix}0&0&0\\0&0&1/3\\0&0&0\end{bmatrix},\quad
\varphi'(d)=\begin{bmatrix}0&0&0\\2/3&2/3&-2/3\\1/3&1/3&-1/3\end{bmatrix},
\]
with~$\lambda'=\begin{bmatrix}15/24&1/2&-1/8\end{bmatrix}$ and~$\gamma'=\begin{bmatrix}1&1&1\end{bmatrix}^t$.
Since each matrix~$\varphi'(a),\ldots,\varphi'(d)$ 
has rank 1, the measure is 1-step Markov.
\end{example}
\section{Identification of Markov sofic measures}\label{sectionIdentification}


We now provide an effective version of Theorem \ref{theoremHolland} by giving an upper bound
for the smallest $k$ such that a sofic measure is $k$-step Markov, if such an integer $k$ exists.
This improves the result given in~\cite{BoylePetersen2011}, in which the bound is much larger (it is expressed as an unbounded pile-up of exponentials).

We prove the following result.
\begin{theorem}\label{theoremNew2}
Let~$\mu$ be an invariant sofic measure and let~$n$ be the dimension of its minimal representation.
 If $\mu$ is $k$-step Markov for some nonnegative integer $k$, then it is $\binom{n}{2}$-step Markov.
\end{theorem}

The proof uses exterior algebra.
Recall that if $E$ is a real vector space, the \emph{exterior algebra} of $E$, denoted $\Lambda E$,
is the largest associative algebra containing $E$ and
generated by $E$
with a product denoted $\wedge$ such that $v\wedge v=0$
for every $v\in V$. The space generated by the homogeneous terms of degree $r$ is denoted by $\Lambda^rE$ (see \cite{Lang2002}).
If $E$ has dimension $n$, the space $\Lambda^rE$
has dimension $n\choose r$.
Every linear map $f\colon E\to F$ extends
to a linear map $\Lambda^r f:\Lambda^rE\to\Lambda^r F$. 
For two real $n \times n$-matrices $M, N$, we have 
$\Lambda^r (MN)=  \Lambda^r (M) \Lambda^r (N)$.

Exterior algebra has been used before
in symbolic dynamics. It is used in Bowen's proof of the rationality of the zeta function of a sofic shift (see \cite{Bowen78}, \cite[p. 123]{beal93}, \cite[Theorem 6.4.8, p. 198]{LindMarcus1995}).

The following lemma is well-known. 

\begin{lemma} \label{lemma.exterior}
A real matrix $M$ has rank less than $r$ if and only if  $\Lambda^r M = 0$. 
\end{lemma}
\begin{proof}
Assume that $\rank(M) \geq r$. Then there are $r$ linearly independent vectors $w_i = Mv_i$. Hence $\Lambda^r M (v_1 \wedge \ldots \wedge v_r) = 
M (v_1) \wedge \ldots M (v_r) = w_1 \wedge \ldots \wedge w_r$.
Since the vectors $w_i$ are independent, $w_1 \wedge \ldots \wedge w_r \neq 0$ and
hence $\Lambda^r M \neq 0$.

Conversely, if $\rank(M) < r$, any family of $r$ vectors $w_i = Mv_i$ are linearly dependent. Hence $\Lambda^k M (v_1 \wedge \ldots \wedge v_k) = M (v_1) \wedge \ldots M (v_k) = 0$.
\end{proof}

\begin{proofof}{of Theorem~\ref{theoremNew2}}
Let~$(\lambda,\varphi,\gamma)$ be the minimal representation of~$\mu$. Let~$S$ be the set of matrices~$\varphi(a)$ for~$a\in A$.
By Theorem~\ref{theoremHolland}, $\mu$ is a~$k$-step Markov measure if and only if
every matrix in~$S^k$ has rank at most $1$. Thus, by Lemma \ref{lemma.exterior}, 
$\mu$ is a~$k$-step Markov measure if and only 
every matrix $M$ in~$S^k$ satisfies $\Lambda^2 M= 0$. 

Let $k$ be the minimal integer such that every matrix $M$ in~$S^k$ satisfies $\Lambda^2 M= 0$.
Let $K_j$ be the left kernel of all matrices $\Lambda^2 M$ for $M$ in~$S^j$ for $0 \leq j \leq k$ with $K_0 = (0, \ldots, 0)$. Then, for $0 \leq j < k$,
$K_j \subseteq K_{j+1}$, and $K_k = \R^m$, where $m = \binom{n}{2}$ is the dimension of the matrices $\Lambda^2 M$.

Let us show that $k \leq m$. Assume that there exist $0 \leq i < j < k$ such that $K_i$ and $K_j$ have the same dimension. Then
$K_i = K_j$. Further, for $v \in \R^m$, $v \Lambda^2 M = 0$ for each $M$ in~$S^{k-j}S^j$. Hence, $\R^m \Lambda^2 N \subseteq K_j$
for each $N$ in~$S^{k-j}$, implying $\R^m \Lambda^2 N \subseteq K_i$ for each $N$ in~$S^{k-j}$, and 
$\R^m \Lambda^2 M = 0$ for each $M$ in~$S^{k-(j-i)}$, a contradiction to the minimality of $k$.
\end{proofof}

An invariant sofic measure is~\emph{$k$-step of rank $r$} if and only if, in its minimal linear representation, each matrix~$\varphi(w)$
corresponding to a word~$w$ of length~$k$ has rank at most $r$.

We have the following generalization of Theorem \ref{theoremNew2}.
\begin{theorem}\label{theoremRankr}
Let~$\mu$ be an invariant sofic measure and let~$n$ be the dimension of its minimal representation.
If $\mu$ is $k$-step of rank $r$ for some nonnegative integer $k$, then it is $\binom{n}{r}$-step of rank $r$.
\end{theorem}
\begin{proof}
Let~$(\lambda,\varphi,\gamma)$ be the minimal representation of~$\mu$. Let~$S$ be the set of matrices~$\varphi(a)$ for~$a\in A$.
By Lemma \ref{lemma.exterior}, $\mu$ is a~$k$-step of rank $r$ if and only 
every matrix $M$ in~$S^k$ satisfies $\Lambda^r M= 0$. 

Let $k$ be the minimal integer such that every matrix $M$ in~$S^k$ satisfies $\Lambda^r M= 0$.
Let $K_j$ be the left kernel of all matrices $\Lambda^r M$ for $M$ in~$S^j$ for $0 \leq j \leq k$ with $K_0 = (0, \ldots, 0)$. Then, for $0 \leq j < k$,
$K_j \subseteq K_{j+1}$, and $K_k = \R^m$, where $m = \binom{n}{r}$ is the dimension of the matrices $\Lambda^r M$. We have all $K_j$ distinct for $0 \leq j \leq k$. 
Hence $k \leq m$.
\end{proof}

Let us show that the bound in Theorem~\ref{theoremNew2} is tight,
even for an alphabet of size $3$.
Let $(e_i)_{1 \leq i \leq n}$ be a basis of $\R^n$, let $m = \lfloor n/2 \rfloor$ and $m^* = \lceil n/2 \rceil$, and let $A = \{a,b,c\}$ be an alphabet of size $3$.
Then, let $\kappa_i = 1/2$ when $i \neq m$, and $\kappa_i = 1$.
We define $\varphi(a)$, $\varphi(b)$ and $\varphi(c)$ as the following linear transformations of $\R^n$:
$\varphi(a)$ sends each vector $e_i$ such that $i < n$ to $\kappa_i e_{i+1}$, and sends $e_n$ to $0$;
$\varphi(b)$ sends each vector $e_i$ such that $m < i < n$ to $\kappa_i e_{i+1}$, sends $e_n$ to $\kappa_n e_1$, and sends other vectors $e_i$ to $0$;
$\varphi(c)$ sends each vector $e_i$ such that $i < m$ to $\kappa_i e_{i+2}$, sends $e_n$ to $\kappa_n e_1$, and sends other vectors $e_i$ to $0$.
By construction of the coefficients $\kappa_i$, the matrix $M = \varphi(a) + \varphi(b) + \varphi(c)$ is stochastic.

Then, $B=\{((e_i \wedge e_j) \mid i < j)\}$ is a
basis of $\Lambda^2 \R^n$ of size $N=\binom{n}{2}$.
Let $\prec$ be the linear order on $\{1,2,\ldots,n-1\}$ defined by
\[1 \prec n-1 \prec 2 \prec n-2 \prec 3 \prec \cdots \prec m^*.\]
It induces a linear order $\lhd$ on $B$, where $e_i \wedge e_j \lhd e_k \wedge e_\ell$ when $j-i \prec \ell-k$ or (\,$j-i = \ell-k$ and $i < k$\,).
Now, each base vector $e_i \wedge e_n$ of $B$ is sent to its successor for $\lhd$ by $\varphi(c)$, if $i < m$, or $\varphi(b)$, if $i > m$; the other two morphisms send it to $0$.
Likewise, each base vector $e_i \wedge e_j$ for which $i < j < n$ is sent to its successor for $\lhd$ by $\varphi(a)$, and to one of its successors or to $0$ by each of $\varphi(b)$ and $\varphi(c)$.
Hence, there exists a word $w$ of length $N-1$ such that $\varphi(w)$ has rank $2$ whereas $\varphi(w)$ has rank $1$ for every word $w$ of length $N$.
Thus, we have built a sofic measure having dimension $n$ which is a $k$-Markov chain with $k={n\choose 2}$ minimal.

Finally, the vectors $\gamma = (1)_{1 \leqslant i \leqslant n}$ and $\lambda = (\lambda_i)_{1 \leqslant i \leqslant n}$ given by $\lambda_i = 2 + (-1/2)^{i-i}$ if $1 \leqslant i \leqslant n/2$ and $\lambda_i = 3$ if $n/2 < i \leqslant n$ satisfy the relations $\lambda = \lambda \cdot M$ and $\gamma = M \cdot \gamma$.
Moreover, the zero entries of $\lambda \cdot \varphi(a^k)$ are the $k$ first entries, and the zero entries of $\varphi(a^k) \cdot \gamma$ are the $k$ last entries.
Hence, both families of vectors $\lambda \cdot \varphi(a^k)$ and $\varphi(a^k) \cdot \gamma$ generate $\R^n$, and the representation $(\lambda / \|\lambda\|_1,\varphi,\gamma)$ is reduced.

\begin{example}

When $n = 5$, the above construction yields the automaton represented in Figure~\ref{figureSharp}.

\begin{figure}[hbt]
\centering
\tikzset{node/.style={circle,draw,minimum size=0.4cm,inner sep=0.4pt}}
\tikzstyle{loop above}=[in=120,out=60,loop]
\tikzstyle{loop right}=[in=30,out=-30]
\begin{tikzpicture}[>=stealth]
\foreach \x in {1,...,5}{
 \node[node](\x)at(2*\x,0){$\x$};
}
\draw[below,->](1)edge node{$\frac{1}{2} a$}(2);
\draw[below,->](2)edge node{$a$}(3);
\draw[below,->,bend right](3)edge node{$\frac{1}{2} a$}(4);
\draw[below,->,bend right](4)edge node{$\frac{1}{2} a$}(5);
\draw[above,->,bend left](1)edge node{$\frac{1}{2} c$}(3);
\draw[above,->,bend left](3)edge node{$\frac{1}{2} b$}(4);
\draw[above,->,bend left](4)edge node{$\frac{1}{2} b$}(5);
\draw[above,->,out=90,in=90,looseness=0.6](5)edge node{$\frac{1}{2} b$}(1);
\draw[below,->,out=-90,in=-90,looseness=0.6](5)edge node{$\frac{1}{2} c$}(1);
\end{tikzpicture}
\caption{A sofic measure which is 10-step Markov.}
\label{figureSharp}
\end{figure}
\end{example}

\begin{remark}\label{remarkSynchro}
Let us consider the particular case where the matrices $\varphi(a)$ have at most one nonzero coefficient in each row.
In this case, $\Lambda^2\varphi(a)$ sends $e_i \wedge e_j$ to $p_{i,a} p_{j,a} (e_k \wedge e_\ell)$
if and only if $\varphi(a)_{ik} = p_{i,a}$ and $\varphi(a)_{j\ell} = p_{j,a}$ with $p_{i,a}, p_{j,a} \neq 0$, 
and thus, the property $\Lambda^2 \varphi(w) = 0$ is independent of the choice of the $p_{i,a}$
(see \cite{James2014}). The property also holds for the $0{-}1$-matrices $\varphi'(a)$ defined by $\varphi'(a)_{ij} = 1$ if $\varphi(a)_{ij} > 0$.
We recover the bound $\binom{n}{2}$
of the synchronization delay of a definite automaton (see \cite{BealSenellart98} or \cite[Exercise 3.4.6]{LindMarcus1995}).
The same result holds if the underlying automaton of $\mu$ is unambiguous, that is, if there is at most one path from $i$ to $j$ with label $w$ for each word $w$. 
\end{remark}

\medskip

The following example shows that the bound is also quadratic with an alphabet of size $2$.

It provides sofic measures of dimension $n$ over a $2$-letter alphabet for which the minimal $k$ such that they are $k$-step Markov is $\frac{1}{4}n^2 + O(n)$, improving 
the lower bound of $\frac{1}{8}n^2$ in \cite{BealSenellart98}.

An \emph{automaton} is given by a finite set of states $Q$ and a set of edges, or transitions $E \subseteq Q \times A \times Q$, where $A$ is a finite alphabet. It is \emph{deterministic} if for each state $x$ and each letter $a$, there is at most one edge $(x, a, y) \in E$.
A word $u$ is \emph{synchronizing} in $\A$ if all paths labeled by $u$ in $\A$ end in a same state (which may depend on $u$);
$\A$ is \emph{definite} if there is an integer $\ell \geqslant 0$ such that all words of length $\ell$ are synchronizing in $\A$.
The smallest such $\ell$ is called the \emph{synchronization delay} of $\A$.

Below, we build a specific two-letter definite automaton $\A$ with $n$ states whose synchronization delay is $n^2/4 + O(n)$.
Let $k = \lfloor n/2 \rfloor$ and let $m$ be the unique odd number in $\{k-1,k\}$.
Let $a$ and $b$ be partial functions defined on $\{1,\ldots,n\}$ to itself, as follows:
$a(x) = x+1$ when $1 \leqslant x \leqslant n-1$; and
\[
b(x) = \begin{cases}
x+3 & \text{when } 3 \leqslant x \leqslant m-2 \text{ and $x$ is odd,} \\
x+1 & \text{when } m+2 \leqslant x \leqslant 2k-5 \text{ and $x$ is odd,} \\
1 & \text{when } 2k \leqslant x.
\end{cases}\]

We prove that the two-letter automaton $\A$ with state space $\{1,2,\ldots,n\}$ and transitions $(x,a,a(x))$ and $(x,b,b(x))$ is definite, and has a synchronization delay of $k^2 + O(k)$.

Proving that $\A$ is definite amounts to proving that the graph $\G_n$ with vertex set $\{x,y\}$, where $1 \leqslant x < y \leqslant n$, and edges $\{x,y\} \to \{a(x),a(y)\}$ (called $a$-edges) and $\{x,y\} \to \{b(x),b(y)\}$ (called $b$-edges), is acyclic.

We say that a pair $\{x,y\}$ is \emph{even} when $y-x$ is even, and \emph{odd} otherwise.
Observe that $a(x) \equiv b(x) \equiv x+1 \pmod{2}$ for all $x$, except when $x = n$ and $n$ is odd, because then $b(n) \equiv n \pmod{2}$.
Hence, each pair $\{x,y\}$ is sent to a pair $\{x',y'\}$ of the same parity, except when $y = n$ is odd and $x' = b(x)$, which means that $\{x,y\}$ was even and $\{x',y'\}$ is odd.
Since no edge transforms an odd pair into an even pair, each cycle consists of even pairs only or of odd pairs only.

In particular, when $n$ is odd, the only edges that leave pairs of the form $\{x,n\}$ 
transform an odd pair into an even pair, so no cycle of $\G_n$ goes through pairs of the form $\{x,n\}$.
Hence, it suffices to prove that the remaining part, i.e., $\G_{2k}$, is acyclic.

Each even pair $\{x,y\}$ is sent to pairs $\{x',y'\}$ such that $x > x'$ only.
Hence, no even pair belongs to a cycle of $\G_{2k}$.
By contrast, each odd pair $\{x,y\}$ has at most one outgoing edge in $\G_{2k}$:
this edge leads to $\{x+1,y+1\}$ if $y < 2k$, and to $\{1,b(x)\}$ if $x < y = 2k$ and $x$ belongs to the domain of $b$.
Hence, when $\delta$ is odd, each maximal path (from $\G_{2k}$) leaving a pair $\{x,x+\delta\}$ leads to $\{2k-\delta,2k\}$, then goes to $\{1,1+\delta'\}$, where $\delta' = b(2k-\delta)-1$.
Since $\delta'$ is the successor of $\delta$ in the order $2k-3 \lhd 5 \lhd 2k-5 \lhd 7 \lhd 2k-7 \lhd \cdots \lhd m \pm 2$, this path is acyclic as well.

Thus, overall, we proved that $\G_{2k}$ (and, when $n$ is odd, $\G_n$) is acyclic.
Furthermore, for all even positive integers $\ell \leqslant m-5$, we have a path
\[\{1,2k-\ell\} \xrightarrow{a^\ell} \{\ell+1,2k\} \xrightarrow{b} \{1,\ell+4\} \xrightarrow{a^{2k-\ell-4}} \{2k-\ell-3,2k\} \xrightarrow{b} \{1,2k-\ell-2\};\]
if $\ell = m-3$, all these steps still exist, except possibly the last one.
Hence, a path of length $2k-2$ joins the states $\{1,2k-\ell\}$ and $\{1,2k-\ell-2\}$.
Concatenating these paths yields a path of length $(k-1)(m-5)$ from $\{1,2k-2\}$ to $\{1,2k-m+3\}$.

Hence, regardless of the parity of $n$, the graph $\G_n$ is acyclic and contains a path of length $(k-1)(m-5) \geqslant (n-3)(n-13)/4$.
 
Now, we define $\varphi(a)$ and $\varphi(b)$ as the following linear transformations of $\R^n$:
for both $\ell = a$ and $\ell = b$, the matrix $\varphi(\ell)$ sends $e_i$ to $\kappa_i e_{\ell(i)}$ when $i \in \mathrm{dom}(\ell)$, and to $0$ otherwise,
where the coefficients $\kappa_i$ are adjusted to make the matrix $M = \varphi(a) + \varphi(b)$ stochastic.
There exists a stochastic vector $\lambda \geqslant 0$ such that $\lambda \cdot M = M$, and since $\A$ has one strongly connected component, $\lambda$ has positive coefficients only.
But then, $\lambda \cdot \varphi(a^s)$ has exactly $s$ zero entries (the final ones), and $\varphi(a^s) \cdot \gamma$ has $s$ zero entries (the first ones), hence $(\lambda,M,\gamma)$ is reduced.
Furthermore, the associated invariant sofic measure is $\ell$-Markov for all $\ell$ large enough (e.g., $\ell \geqslant \binom{2}{n}$), and each such $\ell$ satisfies the inequality $\ell \geqslant (n-3)(n-13)/4$.

\newcommand{\bedge}{edge[decorate,decoration={snake,amplitude=0.3mm,segment length=1.4mm,post length=1mm}] node{$\frac{1}{2}b$}}

\begin{figure}[hbt]
\centering
\tikzset{node/.style={circle,draw,minimum size=0.4cm,inner sep=0.4pt}}
\begin{tikzpicture}[>=stealth]
\foreach \x in {1,...,17}{
 \node[node](\x) at ({-(\x-1)*360/17}:3){\small $\x$};
}
\foreach \x [evaluate=\x as \nextx using int(\x+1)] in {1,2,4,6,7,8} {
 \draw[midway,sloped,auto,swap,->] (\x) edge node {\small $a$} (\nextx);
}
\foreach \x [evaluate=\x as \nextx using int(\x+1)] in {10,12,13,14,15} {
 \draw[midway,sloped,auto,->] (\x) edge node {\small $a$} (\nextx);
}
\foreach \x [evaluate=\x as \nextx using int(\x+1)] in {3,5} {
 \draw[midway,sloped,auto,swap,->] (\x) edge node {\small $\frac{1}{2}a$} (\nextx);
}
\foreach \x [evaluate=\x as \nextx using int(\x+1)] in {16} {
 \draw[midway,sloped,auto,->] (\x) edge node {\small $\frac{1}{2}a$} (\nextx);
}
\draw[midway,sloped,auto,->](17) \bedge(1);
\draw[midway,sloped,auto,swap,bend right=25,->](16)\bedge(1);
\foreach \x/\y in {3/6,5/8}{
 \draw[midway,sloped,auto,->,bend right=40](\x)\bedge(\y);
}
\draw[midway,sloped,auto,->,swap,bend right=40](9)\bedge(10);
\draw[midway,sloped,auto,->,swap,bend left=40](9)edge node{\rotatebox{180}{$\frac{1}{2}a$}}(10);
\draw[midway,sloped,auto,->,swap,bend right=40](11)\bedge(12);
\draw[midway,sloped,auto,->,bend left=40](11)edge node{$\frac{1}{2}a$}(12);
\end{tikzpicture}
\caption{A sofic measure of dimension $17$ which is 36-step Markov.}
\label{figureSharp2}
\end{figure}

\section{Bollobas systems}\label{sectionBollobas}

We present now an approach that establishes a connection with extremal set theory. It 
also leads to an upper bound for the least
integer $k$ such that a sofic measure is $k$-step Markov, but
an exponential one, as we shall see.

Let $E$ be a vector space of dimension $n$
and let $t\ge 0$ be an integer.
A sequence $(U_i,V_i)_{1\le i\le m}$ of subspaces of $E$
is a \emph{Bollobas $t$-system} if $\dim(U_i\cap V_i)\le t$ for $1\le i\le m$ and
$\dim(U_i\cap V_j)>t$ for $1\le i<j\le m$.

The following result from~\cite[Theorem 1.8]{HegedusFrankl2024}, itself a variation on a result of~\cite{Lovasz1977}.
\begin{theorem}\label{theoremHF}
Let~$W$ be an~$n$-dimensional vector space, and let
$(U_i,V_i)_{1\le i\le m}$ be a Bollobas $t$-system.
Then,~$m\leqslant 2^{n-t}$.
\end{theorem}
The following statement gives a connection between sofic measures and Bollobas systems.
\begin{theorem}\label{theoremBollobasSyst}
Let $\mu$ be a sofic measure, let $(\lambda,\varphi,\gamma)$ be its minimal representation,
and let $n$ be its dimension.
Assume that $\mu$ is $k$-step Markov for some $k\ge 1$. 
For $r\le n-1$, let $w_1,\ldots,w_m\in A^*$ be such that
$\varphi(w_1\cdots w_m)$ has rank $r$ with $m$ as large as possible. Then $(U_i,V_i)_{1\le i\le m}$ with
\[U_i=\ker(\varphi(w_{i+1})),\quad V_i=\Im(\varphi(w_i))\]
for $1\le i\le m-1$ is a Bollobas $r$-system.
\end{theorem}
\begin{proof}
Set $M_i=\varphi(w_i)$.
Since~$\rank(M_iM_{i+1})=r$, we have~$U_i\cap V_i=\{0\}$.
Furthermore, when~$1\leqslant i<j\leqslant k-1$, if~$U_i\cap V_j=\{0\}$,
then~$M_1\cdots M_j(M_i\cdots M_j)M_{j+1}\cdots M_k$ has rank~$r$, in contradiction with the maximality of~$k$.
Thus,~$\dim(U_i\cap V_j)>0$. 
\end{proof}

\begin{corollary}\label{corollaryNew1}
Let~$S$ be a set of~$n\times n$-matrices of rank~$r$ such that any long enough product of elements from~$S$ has rank at most~$r-1$.
Then, every product of~$2^{n+1}$ elements from~$S$ has rank at most~$r-1$.
\end{corollary}
\begin{proof}
Let~$M_1\cdots M_k$ be a product of matrices~$M_i\in S$ of rank~$r$, with~$k$ maximal.
For all $i \leqslant k-1$, let~$U_i=\ker(M_{i+1})$ and~$V_i=\Im(M_{i})$:
since~$\rank(M_iM_{i+1})=r$, we have~$U_i\cap V_i=\{0\}$.
Furthermore, when~$1\leqslant i<j\leqslant k-1$, if~$U_i\cap V_j=\{0\}$,
then~$M_1\cdots M_j(M_i\cdots M_j)M_{j+1}\cdots M_k$ has rank~$r$, in contradiction with the maximality of~$k$.
Thus,~$\dim(U_i\cap V_j)>0$. 

By Theorem~\ref{theoremHF} applied with~$m=k-1$ and~$t=0$, we obtain~$k-1\leqslant 2^n$,
and therefore~$k\leqslant 2^{n+1}$.
\end{proof}

\begin{proofof}{of the bound $k =2^{n^2-1}$}
Let~$(\lambda,\varphi,\gamma)$ be the minimal representation of~$\mu$. Let~$S$ be the set of matrices~$\varphi(a)$ for~$a\in A$.
Let~$k$ be the minimal integer such that~$\mu$ is a~$k$-step Markov measure. By Theorem~\ref{theoremHolland}, every matrix in~$S^k$ has rank at most one. Therefore, applying~$n-1$ times Corollary \ref{corollaryNew1},
we obtain~$k\leqslant (2^{n+1})^{n-1}=2^{n^2-1}$.
\end{proofof}
The bound given by Theorem~\ref{theoremHF} is sharp. Therefore, the Bollobas systems given by Theorem~\ref{theoremBollobasSyst} are particular.

\section{Use of AI}
We have used AI systems to produce the results of this paper.
In a first step, we have used Gemini to find the link with
extremal set theory leading to the results of Section~\ref{sectionBollobas}. In a second step, the use of 
OpenAI GPT-6 Astra was crucial to suggest the use of exterior algebra to obtain Theorem~\ref{theoremNew2}.
\bibliographystyle{plain}
\bibliography{probas}
\end{document}